\documentclass[10pt]{amsart}
\usepackage{amssymb, amscd, amsmath, amsthm, epsf, epsfig, latexsym}

\def\zz{{\bf Z}}
\def\ff{{\bf F}}
 
\def\qq{{\bf Q}}
\def\cc{{\bf C}}

\def\co{\colon\thinspace}
\def\cs{\mathop{\#}}

\def\calk{\mathcal{K}}

  \def\call{\mathcal{L}}
  \def\calb{\mathcal{B}}
 \newcommand\spin{\mathfrak s}
  \newcommand\spint{\mathfrak t}

\newcommand{\Sign}{\operatorname{Sign}}
\newcommand{\Arf}{\operatorname{Arf}}
\newcommand{\lk}{\operatorname{lk}}

\newtheorem{theorem}{Theorem}
\newtheorem{lemma}[theorem]{Lemma}
\newtheorem{corollary}[theorem]{Corollary}

\theoremstyle{definition}
\newtheorem{definition}[theorem]{Definition}
\def\co{\colon\thinspace}

\numberwithin{equation}{section}

\begin{document}

\title{The nonorientable four-genus of knots}
\author{Patrick M. Gilmer}
 \author{Charles Livingston}
 
 \thanks{The first author was partially supported by  NSF-DMS-0604580,  NSF-DMS-0905736}  
 \thanks{The second author was partially supported by  NSF-DMS-0707078, NSF-DMS-1007196}

 \address{Patrick Gilmer: Department of Mathematics, Louisiana State University, Baton Rouge, LA 70803}
\email{gilmer@math.lsu.edu}
 
\address{Charles Livingston: Department of Mathematics, Indiana University, Bloomington, IN 47405 }
\email{livingst@indiana.edu}


 \begin{abstract}  
 We develop obstructions to a knot $K \subset S^3$ bounding a smooth punctured Klein bottle  in $B^4$.  The simplest of these is based on the linking form of the 2--fold branched cover of $S^3$ branched over $K$.  Stronger obstructions are  based on the Ozsv\'ath-Szab\'o correction term in Heegaard-Floer homology, along with the $G$--signature theorem and the Guillou-Marin generalization of Rokhlin's theorem.  We also apply Casson-Gordon theory to  show that for every $n>1$ there exists a knot that does not bound a topologically embedded {\it ribbon} nonorientable surface $F$ in $B^4$ with first Betti number $\beta_1 (F) <  n$.  
 
  \end{abstract}

\maketitle

Since  the   1960s, steady progress has been achieved in determining the 4--genus of knots in $S^3$.  Highlights   include the early work   of Fox-Milnor~\cite{f,fm} and Murasugi~\cite{murasugi}, advances made by Tristram~\cite{tristram} and Levine~\cite{le2}, and the development of Casson-Gordon theory~\cite{cg}.  Freedman's work on 4--dimensional topological surgery opened up our understanding of the 4--genus in the topological category~\cite{freedman-quinn}.  The introduction of gauge theory and Seiberg-Witten theory led to further progress, culminating with the Kronheimer-Mrowka  proof of the Milnor conjecture related to the 4--genus of torus knots~\cite{km2}. 
Cochran, Orr and Teichner~\cite{cot}  extended the slice obstructions of Levine and Casson-Gordon into an infinite sequence of obstructions. 
The application of Heegaard-Floer theory by 
 Ozsv\'ath-Szab\'o~\cite{os3} and of Khovanov homology
 by Rasmussen~\cite{ra} has led to a series of new developments.

In contrast to this, our understanding of which nonorientable surfaces in the 4--ball can bound a given  knot has been extremely limited.
For a connected surface $F$ with nonempty connected boundary, let $h(F) =  \beta_1(F)$, the first Betti number: $\beta_1(F) = \dim H_1(F, \qq) $.  If $F$ is orientable   of genus $g$, then $h(F) = 2g$.   Let $h(K) = \min\{h(F)\ |\ \partial F = K\}$, where the minimum is taken over  smoothly embedded nonorientable surfaces in $B^4$ bounded by $K$.  We call this the  nonorientable 4--genus of $K$.  The nonorientable 4--genus clearly  behaves much differently than the 4--ball genus; as an example, the  $(2,2n+1)$--torus knot has 4--ball genus $n$, but nonorientable 4--genus 1.
  
In 1975,  Viro~ \cite{viro2} proved that the figure eight knot, $4_1$,  cannot bound a Mobius band in $B^4$. His method was to study the Witt class  of the intersection forms of 4--manifolds arising as the branched cover of the 4--ball branched over nonorientable surfaces.  
 Twenty years later, Yasuhara~\cite{yasuhara}  applied  Guillou and  Marin's   ~\cite{gm} generalization of Rokhlin's theorem (concerning the signature modulo  sixteen of closed smooth spin 4--manifolds) to formulate an obstruction to a knot bounding a Mobius band; this obstruction is based on the signature and Arf invariant of the knot  and was applied to the figure eight and granny knot, $3_1 \cs 3_1$. See Theorem \ref{ythm} below.  

We let $(H_1(M(K)), \lk)$ denote the linking form on the $2$--fold branched cover of $S^3$ branched over $K$. 
In 2000, using a result of the first author ~\cite{gi3}, Murakami and Yasuhara  \cite[Theorem 2.5]{my} found an obstruction, in terms of $(H_1(M(K)), \lk)$,  to a knot bounding a 
nonorientable surface of genus $g$;  see Theorem \ref{mythm}.  The only application of this theorem given in \cite{my}  was that $h(4_1)=2$.  We will use Theorem \ref{mythm} to prove  Theorem \ref{theoremklein},   which gives some interesting examples of knots with $h(K) \ge 3.$ In particular,
we show:

\vskip.1in
\noindent{\bf Theorem A.}$$h(4_1 \cs 5_1) = 3.$$
\vskip.1in

\noindent But we will also find 
 knots not bounding punctured Klein bottles  for which the theorem of Murakami and Yasuhara does not provide an obstruction:

\vskip.1in
\noindent{\bf Theorem B.}
{\it There exist knots $K$ such that  $h(K) \ge 3$ and $(H_1(M(K)), \lk)$  has a presentation of rank $2$.}\vskip.1in  

\noindent (As described in Appendix~\ref{alglK}, such a presentation of $(H_1(M(K)), \lk)$ consists of a $2\times 2$ integer matrix $A$ such that $H_1(M(K)) \cong \zz^2/A\zz^2$ and the linking form is given by $A^{-1}$ with respect to an appropriate generating set.)
\vskip.1in

 Let $h^r(K) = \min\{h(F)\ |\ \partial F = K \text{ and $F$ is ribbon and nonorientable} \}$.   This is called the nonorientable ribbon genus.
  Conjecturally, $h^r(K) = h(K)$ for all $K$.  We prove:

\vskip.1in
\noindent{\bf Theorem C.}
{\it For every $N$, there exists a knot $K$ such that $h^r(K) \ge  N$.}\vskip.1in

One set of examples is built from a particular knot,  
$D_{6}$, the $6$--twisted double of the unknot.  Specifically, if $K = n D_{6}  = \partial F$, 
where $F$ is a nonorientable ribbon surface bounded by $K$, then  $h(F) \ge \frac{n}{2}$.   This knot is of particular interest in that it was the first example of an algebraically slice knot that is not slice, discovered by Casson and Gordon~\cite{cg}.     
  \vskip.1in

\noindent{\bf Outline}
We will work in the smooth category until Section~\ref{topsection}, where we describe the extent to which our work applies in the topological category and provide examples distinguishing the two categories with respect to the nonorientable genus.

In Section~\ref{linkobstruct},  we give the proof of the Murakami-Yasuhara theorem and  we use this  to obstruct  knots from bounding  punctured Klein bottles in $B^4$.   Since the obstruction is purely homological, it cannot detect the distinction between the smooth and topological categories.

In Section~\ref{yasuharasection},  we present a short proof of Yasuhara's formula which relates the signature and  the Arf invariant  of a knot  that bounds a Mobius band  in $B^4$.   Note that   the derivation (both Yasuhara's and ours) depends on a generalization of   Rokhlin's Theorem  due  to Guillou and Marin~\cite{gm}, and thus it holds only in the smooth category. 
 
There are three steps to developing stronger obstructions to a knot bounding a  Klein bottle.   For any knot $K$, let $M(K)$ denote the 2--fold branched cover of $S^3$ branched over $K$.  If $K$ bounds a punctured Klein bottle  $F \subset B^4$, then $M(K)$ bounds a smooth, compact, oriented 4--manifold $W$ with second Betti number $\beta_2(W) = 2$.
Here we take for  $W$  
  the 2--fold branched cover of $B^4$ branched over $F$,
  which we denote $W(F)$.
    In the first step, Section~\ref{sectionposdef}, we 
   provide conditions on $K$ that ensure that  $W(F)$ cannot have a
  definite intersection form;  these are based solely on the signature and Arf invariant of $K$. This argument expands upon our derivation of Yasuhara's formula.

  In Section~\ref{indefinitesection} we  present obstructions to   $M(K)$ 
  bounding  a $W$
  with the   properties above and with  an
  indefinite intersection form; these are based on the linking form of $M(K)$. 
   In the third step, presented in Section~\ref{sectionnegdef}, knots are modified without altering previously established properties so that the new  $M(K)$ does not bound such a $W$ with negative definite intersection form;  this work uses the correction term of Heegaard-Floer homology~\cite{os2}.  We conclude work surrounding Theorem~B    in Section~\ref{sectionexplicitkknot}, where we construct examples for which all three  obstructions apply.

Section~\ref{cgsection} presents background in Casson-Gordon theory and Section~\ref{thmbsection} is  devoted to the proof of Theorem C.  Section~\ref{topsection} discusses the topological category.  In particular, we prove:
\vskip.1in
\noindent{\bf Theorem D} {\it There exists a knot $K$ that bounds a topological Mobius band in $B^4$, but does not bound a smooth Mobius band.}\vskip.1in

The appendix provides  background concerning intersection forms of 4--manifolds and  linking forms.

 \vskip.1in
\noindent{\bf Nonorientable surfaces in $S^3$} There is a notion of the nonorientable 3--genus of a knot, sometimes called the {\it crosscap number}.  This clearly provides an upper bound for the nonorientable 4--genus, but one that can be quite weak.  We wish to note here some of the relevant research on this topic.  Basic foundational work was begun by Clark in~\cite{clark}.  Classes of knots for which there are good results include torus knots~\cite{mattsize, tera}, 2--bridge knots~\cite{hira-tera}, and pretzel knots~\cite{ichi-miz}.  The interplay between the nonorientable 3--genus and knot concordance was studied in~\cite{liv, zhang}.

 \vskip.1in 
\noindent{\bf Acknowledgements}  Many people have helped us in this work, providing us with important background material and stimulating conversation.  In particular we wish to thank Jim Davis, Allan Edmonds, Matt Hedden, Paul Kirk,   and Peter Kronheimer. Special thanks go to  Oleg Viro for his help in understanding his early work studying nonorientable surfaces and knot theory.  We also thank the referee for contributions that greatly improved the exposition.

\section{Murakami-Yasuhara Theorem and Consequences}\label{linkobstruct}
The main new result in this section is  that the  linking form of the 2--fold branched cover of a knot provides obstructions to a knot bounding a punctured Klein bottle in $B^4$. 

\begin{lemma}\label{bettilemma} Let $K\subset S^3$ bound a connected surface $F\subset B^4$ and let $W(F)$ be the 2--fold branched cover  of $B^4$ branched over $F$.  Then $\beta_2(W(F)) = \beta_1(F)$.

\end{lemma}

\begin{proof}   The $2$--fold branched cover of $S^3$ branched along $K$, $M(K)$,  is a rational homology sphere.  We have that $H_1(W(F),\qq) = 0$ (see, for instance,~\cite[Lemma 2]{massey}) and from duality and the long exact sequence of the pair $(W(F),M(K))$, $H_3(W(F), \qq) = 0$.   Since $F$ is smooth, it can be viewed as subcomplex of a triangulation of $B^4$.  Counting simplices  in the base and  the cover we have $\chi(W(F)) = 2 \chi(B^4)- \chi(F) = 2 - (1 - \beta_1(F))$.  Since $\chi(W(F)) = 1 +\beta_2(W(F))$, it follows that 
 $\beta_2(W(F)) = \beta_1(F)$, as desired.
\end{proof}

\begin{theorem}[Murakami-Yasuhara]  \label{mythm}  Let $K \subset S^3$ be a knot.  The linking form $(H_1(M(K)), \lk)$ splits as a direct sum $(G_1, \beta_1) \oplus (G_2, \beta_2)$ where $(G_2, \beta_2) $ is metabolic and $(G_1, \beta_1)$ has a presentation of rank $h(K)$. 
\end{theorem}

\begin{proof}  Let $W$ be a compact 4--manifold with connected boundary $\partial W$  satisfying $H_1(\partial W, \qq) = 0$.  According to~\cite{gi3},  $( H_1(\partial W), \lk)$ splits as a direct sum $(G_1, \beta_1) \oplus (G_2, \beta_2)$ where $(G_2, \beta_2)$ is metabolic and $(G_1, \beta_1)$ is presented by a matrix representing the intersection form of $W$.  Details of this result are presented as   Lemma~\ref{githm} in the appendix.   Combining this with Lemma~\ref{bettilemma} yields a proof of the theorem. (This argument is  the same as in the proof given in \cite{my}.)
\end{proof}

Here is an application of this theorem along the lines of \cite[Example 2.8]{my} which discussed the figure eight.

\begin{corollary}\label{theoremmobius} Suppose that $H_1(M(K)) = \zz_n$ where $n$ is the product of primes, all with odd exponent.  Then if $K$ bounds a Mobius band in $B^4$, there is a  generator $a \in H_1(M(K))$ such that $\lk(a,a) = \pm 1/n$.
\end{corollary}

\begin{proof} One has that $(H_1(M(K)), \lk)$ splits as a direct sum $(G_1, \beta_1) \oplus (G_2, \beta_2)$ where $(G_2, \beta_2) $ is metabolic and $(G_1, \beta_1)$ has a presentation of rank one. Since each prime occurs with odd exponent, no primary summand of $H_1(M(K))$ is metabolic, and thus $(H_1(M(K)),\lk)$ has a presentation of rank one.  But that presentation matrix must be of the form $(\pm n)$, so the linking form on $H_1(M(K))$ is given by $ {\pm 1}/{n}$.  That is, some generator of $H_1(M(K))$ has self-linking $\pm 1 /n$, as desired.
\end{proof}

For a nonsingular bilinear form on $\zz_p^n$, the discriminant is, by definition, given by  $(-1)^{n(n-1)/2} \det(Q) \in  \zz_p^*/(\zz_p^*)^2$, where $Q$ is an $n \times n$ matrix with entries in $\zz_p$ representing the form.  Details can be found in the appendix.

\begin{theorem}\label{theoremklein} Suppose that $H_1(M(K)) = \zz_p \oplus \zz_p$ where $p$ is   prime.  Then if $K$ bounds a punctured Klein bottle  in $B^4$, the discriminant of the linking form is $\pm 1 \in \ff_p^* / (\ff_p^*)^2$.

\end{theorem}

\begin{proof} Consider a possible splitting $(H_1(M), \lk) = (G_1, \beta_1) \oplus (G_2, \beta_2)$.  No nonsingular form on $\zz_p$ is metabolic, so either $H_1(M) = G_2$ (in which case we immediately conclude that the discriminant of the linking form is 1), or $H_1(M) = G_1$.

In the second case, $H_1(M) = G_1$, $H_1(M)$ is presented by a $2 \times 2$ symmetric matrix.  Since the matrix presents $\zz_p \oplus \zz_p$, each entry is divisible by $p$, so the matrix can be written as 
 \[
\left(
\begin{array}{cc}
  pa&   pb   \\
 pb &  pc   \\
\end{array}
\right)
\]
 where $ac-b^2 = \pm 1$.
The linking form is then represented by the inverse of this matrix.  Viewing the linking form as taking values in $\zz_p$, the linking form is represented by the matrix
\[
\pm\left(
\begin{array}{cc}
  c&   -b   \\
  -b &  a   \\
\end{array}
\right).
\]
For a $2 \times 2$ matrix, the discriminant is given by the negative of the determinant.  This completes the proof.
\end{proof}

\begin{proof}[Proof of Theorem A] Consider the knot $4_1 \cs 5_1$, which can be written as the connect sum of 2--bridge knots $K_{5/2} \cs K_{5/1}$.  The linking form, when viewed as a  bilinear form on a $\zz_5$ vector space is given by a diagonal matrix with diagonal $[{-2},{-1}]$, and so has discriminant $-2 \in  \ff_5^* / (\ff_5^*)^2$.  Since $2$ is not a square, or minus a square, in $\zz_5$, we conclude that this knot cannot bound a Klein bottle in $B^4$.  
 As $4_1$ has genus one, and $5_1$ bounds a Mobius band in $S^3$, we see that $h(4_1 \cs 5_1)=3$.
\end{proof}

\section{Yasuhara's formula}\label{yasuharasection}
 The goal of this section is to present a proof of Yasuhara's formula, as stated in the next theorem.
 Let $\sigma(K)$ denotes the Murasugi signature of $K$   and let {$\Arf(K)$} denote the $\zz_2$--valued Arf invariant of $K$;  4$\Arf(K)$ is defined via the natural inclusion of $\zz_2$ into $\zz_8$. 

\begin{theorem}(Yasuhara)\label{ythm}
If a knot $K \subset S^3$ bounds a  smoothly embedded Mobius band in $B^4$, then $\sigma(K) + 4 \Arf(K)  \equiv  0 \text{ or } \pm 2 \mod 8$. \end{theorem}

The proof follows quickly from two results. The first (Theorem \ref{glthm})  was proved by Gordon-Litherland~\cite{gl} extending earlier work of Kauffman-Taylor~\cite{kt} and Viro~\cite{viro} that gave a 4--dimensional interpretation of the Murasugi knot signature~\cite{murasugi}. 
 The second (Theorem \ref{gmthm}),
 suggested by ~\cite{gi4,gi5}, 
is a consequence of the Guillou-Marin~\cite{gm} generalization of Rokhlin's theorem \cite{rh1} (see also Rokhlin~\cite{rh2}, Freedman-Kirby~\cite{fk}, Matsumoto~\cite{mats}, and Kervaire-Milnor~\cite{km}).    Details of the definitions and properties of the terms that appear in each  will be provided following the proof of Theorem~\ref{ythm}. 
 
 In summary, let $K$ be a knot bounding a smooth connected  surface $F \subset B^4$.  Let $\sigma(K)$ and $\Arf(K)$ denote the signature and Arf invariant, taking values in $\zz$ and $\zz_2$, respectively.  Let $W(F)$ be the 2--fold branched cover of $B^4$ branched over $F$, having signature $\Sign(W(F))$.  Let $F \cdot F$ be the self-intersection number of $F$ 
 which is always even.  Let $\beta(B^4, F)$ denote the Brown invariant~\cite{br}, $\beta$, of the Guillou-Marin form $q_F:H_1(F,\zz_2) \rightarrow \zz_4.$

\begin{theorem}\label{glthm} $\sigma(K) = \Sign({W(F)}) +  \frac{1}{2}F\cdot F $.
\end{theorem}

\begin{theorem}\label{gmthm}  If $F$ is smooth, then $4\Arf(K)   \equiv    \beta(B^4 ,F)+  \frac{1}{2} F \cdot F  \mod 8$.  
\label{arf}
\end{theorem}

Subtracting these equations and using $-4\Arf(K) =  4\Arf(K) \mod 8$, we obtain:
\begin{corollary}\label{Fcor} $\sigma(K) + 4\Arf(K) = \Sign({W(F)})-\beta(B^4 ,F)  \mod 8$.
\end{corollary}

 \begin{proof}[Proof of Theorem \ref{ythm}]   
 Now assuming $F$ is a Mobius band, by Lemma~\ref{bettilemma}  we have that $H_2(W(F),\qq) = \qq$.    It follows that $\Sign(W(F)) = \pm1$.  It follows from the definition of the Brown invariant for a quadratic form $q_F$, that  if the form is one dimensional and odd (as it must be for a Mobius band), then $\beta(B^4 ,F) =\beta(q_F)  \equiv  \pm 1\mod 8$. The result now follows from Corollary \ref{Fcor}.
\end{proof}
  
\subsection{Background and proof summary for Theorem~\ref{glthm}}$  $ 

With regards to the Gordon-Litherland theorem, we leave most of the details to~\cite{gl}.  The core result is a special case of the $G$--signature theorem~\cite{as}, which holds in the topological category:   If $\tau$ is an orientation preserving involution of an oriented closed 4--manifold $W$ with fixed set a surface $\widetilde{F}$, then  
$ 2\Sign(\overline{W}) -\Sign(W)  = \widetilde{F} \cdot \widetilde{F}$, 
where $\overline{W}$  is the quotient of $W$ under the action of $\tau$.  The self-intersection of the possibly nonorientable surface $\widetilde{F}$ is defined by counting the self-intersection points of a transverse push-off of $\widetilde{F}$ with itself, using a local orientation on $\widetilde{F}$ to give a well-defined sign to each intersection point.  
 For a topological proof  of the $G$--signature theorem in this setting, see~\cite{gi2, gordon}.

 In the case that ${F}$ is a possibly nonorientable connected  surface in the $4$--ball with  nontrivial boundary in $S^3$, the self-intersection is defined by counting intersections with a push off $\hat F$ of $F$ as above, except one must pick $\hat{F}$ so that  $\partial F$ and $\partial \hat{F}$ have linking number zero. 
Now, $F$ also has a disjoint push-off, say $F'$, and  $F \cdot F$ can then be seen to be
  {\em minus} the linking number of $\partial F$ and $\partial {F'}$ coherently oriented. One has that $F \cdot F$ is even. This follows from the fact that  every closed surface in $S^4$ has even self-intersection.   (Modulo two,  $F\cdot F $ represents the self-intersection of $[F]$, where $[F]$ is the class in $H_2(S^4, \zz_2)$ represented by $F$. But this is clearly trivial, since $H_2(S^4,\zz_2) = 0$.)
 
The proof of Theorem~\ref{glthm} proceeds as follows.  From the additivity of signature and the $G$--signature theorem, the quantity   
 $\Sign({W}) + \widetilde{F}\cdot \widetilde{F}  $ is independent of the choice of $F$.  Next it is observed that $ \widetilde{F}\cdot \widetilde{F} = \frac{1}{2}F\cdot F$.  Finally, if $F$ is a Seifert surface for $K$ pushed into the 4--ball, the value of
  $\Sign({W(F)}) + \frac{1}{2} {F}\cdot{F}  $ is the Murasugi signature $\sigma(K)$.

\subsection{Background and proof summary for Theorem~\ref{gmthm}}$  $

\subsubsection {\bf The Arf invariant} One may use the Seifert form of a knot with Seifert surface $F \subset S^3$,  $V\co H_1(F) \times H_1(F) \to \zz$, to  define a quadratic form   $\frak{q}_F\co H_1(F,\zz_2)   \to \zz_2$, by $\frak{q}_F(x)= V(\hat x,\hat x) \mod 2$ where $\hat x \in H_1(F)$ reduces to  $x \in H_1(F, \zz_2)$. The form is quadratic with respect to the intersection form of $F$:  $\frak{q}_F(x+y) = \frak{q}_F(x) + \frak{q}_F(y) + \left<x, y\right>$. The Arf invariant of $K$, denoted $\Arf(K)$, equals either $0$ or $1 \mod 2$, depending on whether this form takes value 0 or 1 on a majority of the elements, resepectively.

\vskip.05in

\subsubsection {\bf The Brown invariant} Let $V$ be a finite dimensional $\zz_2$ vector space with nonsingular inner product $\left< \cdot,\cdot \right>\co V \times V \to \zz_2$.  A function $q\co V \to \zz_4$ is called a quadratic form with respect to $\left< \cdot , \cdot \right>$  if $q(x+y) = q(x) + q(y) + 2\left<x,y\right>$ for all $x, y \in V$. 
The $2$ in $2\left<x,y\right>$   denotes the injective map  from $\zz_2$ to $\zz_4$.
   To such a form there is a well-defined $\zz_8$--valued invariant $\beta(q)$, defined by the formula $\beta(q) = \frac{1}{ {2}^{n/2}}\sum_{x \in V} i^{ q(x)} $.  (Here $i = \sqrt{-1}$.)  This sum can be shown to be an eighth root of unity; we identify the multiplicative group of eighth roots of unity with the additive group   $\zz_8$.   We will use the fact that $\beta( q_1 \oplus q_2) = \beta( q_1)  +\beta(  q_2)$.  
   
   An alternative viewpoint of the Brown invariant that applies in the singular setting is nicely presented in~\cite{kirbymelvin}. \vskip.1in

\subsubsection{\bf The Guillou-Marin form, $q_F$.}

Let $F$ be a characteristic embedded closed surface  in a 4--manifold $W$ where $H_1(W) = 0$.  Characteristic means $F\cdot G  \equiv  G\cdot G \mod 2$ for all surfaces $G$ embedded in $W$, where $F\cdot G$ represents the intersection number of $F$ and $G$, computed by isotoping $G$ to be transverse to $F$ and counting the number of points of intersection.  
  In~\cite{gm}, (see also~\cite{mats}),  Guillou-Marin   define for such a surface a quadratic form $q_F\co H_1(F,\zz_2) \to \zz_4$  which is quadratic with respect to the intersection pairing on  $H_1(F,\zz_2)$.  Here is a summary of the definition of $q_F$.   This definition is not used in our proofs.

 For a class $z \in H_1(F,\zz_2)$, let $\alpha$ be an {\it oriented} embedded curve  on $F$ which represents 
 $z$. The normal bundle to $\alpha$ in $W$, denoted $N_\alpha$, is a trivial oriented 3--dimensional vector bundle over $\alpha$, i.e. $T(W)=T(\alpha) \oplus N_\alpha$ as oriented bundles.  Let $\alpha$ bound a connected oriented surface $G$ embedded in $W$,  transverse to $F$  along $\alpha$ and in its interior.

 A nontrivial section to $N_\alpha$ is given by the inward 
 pointing normal to $\alpha$ in $G$.  This splits $N_\alpha$ as $\epsilon \oplus N'_\alpha$, where $N'_\alpha$ is a trivial oriented 2--dimensional vector bundle over $\alpha$.  Note that the normal bundle to $\alpha$ in $F$ can be viewed as a subbundle of $N'_\alpha$.  
  
 The normal bundle to $G$ in $W$ has a nonvanishing section  which restricts to give a nonvanishing section of $N'_\alpha$. This section of $N'_\alpha$ is unique up to homotopy.  Since $N'_\alpha$ is trivial and oriented, this section determines a trivialization of $N'_\alpha$.  The twisting number $t(\alpha, F, G)$ is the number of half-twists $F$ makes with respect to this trivialization.
  
 The Guillou-Marin form is defined by $q_F(z) = t(\alpha,F,G) + 2 G\cdot F + 2 z \cdot z \mod 4$.
 
\subsubsection {\bf The Brown invariant of a pair $(W,F)$.}  $  $

 Let $W$ be a closed 4--manifold with $H_1(W) = 0$.  Let $F$ be a closed surface embedded in $W$ which is characteristic.  The  Brown invariant of $q_F$ is denoted by  by $\beta(W,F)$.

\begin{theorem}[Guillou-Marin]\label{guilloumarinthm}
If $W$ is a closed 4--manifold with $H_1(W) = 0$ and $F\subset W$ is a characteristic surface, then $$\Sign(W) \equiv  2 \beta(W,F) +
F\cdot F  \mod 16.$$
\end{theorem}

\subsubsection{\bf  The Arf invariant and nonorientable surfaces}
Let $K$ be a knot in $S^3$ bounding a possibly nonorientable surface $F$ in $B^4$.  We can define a form  $q_F$ by the same procedure as used by Guillou-Marin for  closed surfaces. Let $\beta(B^4, F)$ be the Brown invariant of $q_F$. Then we can define an invariant by 
$A(K)  \equiv   \beta(B^4, F) + \frac{1}{2} F\cdot F   \mod 8$.
That this is well-defined, depending only on $K$, follows immediately from Theorem~\ref{guilloumarinthm}.
If $F$ is  an orientable Seifert surface for $K$, it can be pushed into $B^4$.  For that $F$, $F \cdot F = 0$. If we represent a class in $H_1(F,\zz_2)$ by a curve $c$, then 
$q_F(c) = 2 \frak{q}_F(c) \in \zz_4$. Thus, $i^{q_F(c)} = \pm 1$ depending on whether $\frak{q}_F(c)$ is zero or non-zero.  It follows that  $\beta(q_F)$ is positive or negative depending on whether a majority of elements have $\frak{q}_F(c)$   zero or non-zero.  Knowing that $\beta$ is a unit complex number tells us that $\beta(q_F)$ is $1$ or $-1$.  Switching to additive notation, $\beta(q_F) = 0$ or $4\mod8$ depending on whether $\frak{q}_F$ is zero or non-zero on a majority of elements of $H_1(F,\zz_2)$.  

We conclude from the previous paragraph that  
 the invariant $A(K)$ agrees with $\Arf(K)$ followed by  the inclusion of $\zz_2$ into  $\zz_{8}$. 
Thus, we have Theorem~\ref{arf}.
  
\section{Positive Definite forms} \label{sectionposdef}

\begin{theorem}\label{pdefthm} If $K$ bounds a punctured Klein bottle $F$ in $B^4$ and $W(F)$, the 2--fold branched cover of $B^4$ branched over $F$, has a positive definite intersection form, then $\sigma(K) + 4 \Arf(K)  \equiv  0, 2, \text{or } 4 \mod 8$.  If $W(F)$ is negative definite then $  \sigma(K) + 4 \Arf(K)  \equiv  0, 4, \text{or } 6 \mod 8$

\end{theorem} 
\begin{proof}
 Consider the positive definite case.
Lemma~\ref{bettilemma} implies that $\beta_2(H_2(W(F)) = 2$, and hence,  $\Sign(W(F))=2.$ The intersection form on the Klein bottle is diagonalizable; using this, a quick computation shows that there are exactly  four  $\zz_4$--valued quadratic forms on $H_1(F, \zz_2)$. When one computes the  Brown invariant for each of these,  the result is  either $0$ or $\pm 2 \in \zz_8$.  Corollary \ref{Fcor} then yields the statement of the theorem in the positive definite case.
 
 The negative definite case can be handled similarly, or can be proved by taking mirror images.
\end{proof}

\section{Indefinite Forms}\label{indefinitesection}

\begin{theorem}\label{indefthm}Suppose that $H_1(M(K)) = \zz_p \oplus \zz_p \oplus \zz_q$ where $ q \equiv 1 \in \ff_p^* / (\ff_p^*)^2$.   
If $H_1(M(K))$ is the boundary of a 4--manifold $W$ with  second Betti number 2 which has  an indefinite intersection form, then the linking form restricted to  $\zz_p \oplus \zz_p \subset H_1(M(K))$ is metabolic.
\end{theorem}

\begin{proof} 
By Lemma~\ref{githm} in Appendix~\ref{giapp},
 {\em minus}
 the linking form on $H_1(M(K))$ can be written as $\beta_a \oplus \beta_m$, defined on the direct sum of groups $G_a \oplus G_m$, where $\beta_m$ is metabolic and $\beta_a$ is presented by $A$, the intersection matrix of $W$.   

Since a linking form on $\zz_p$  cannot be metabolic, the $\zz_p \oplus \zz_p$ summand of $H_1(M(K))$ is either entirely contained in $G_a$ or in $G_m$.   In the second case it is automatically metabolic, so we focus   $\zz_p \oplus \zz_p \subset  G_a$. 

  Similarly $\zz_q$ cannot be metabolic, so $\zz_q \subset G_a$. Thus, we are in the case that  $G_m$ is trivial and $G_a = \zz_p \oplus \zz_p \oplus \zz_q$, which we can abbreviate $G_p \oplus G_q$.

The matrix $A^{-1}$ represents the linking form on $\zz_p \oplus \zz_p \oplus \zz_{q}$ with respect to a pair of generators, say $x$, and $y$.  The $G_p$ summand of $\zz_p \oplus \zz_p \oplus \zz_{q}$ is generated by $qx$ and $qy$.  The linking form restricted to this summand is thus represented by the matrix $ q ^2 A^{-1}$.  

To convert the linking matrix into a matrix with $\ff_p$ entries, we multiply all the entries by $p$.  (Necessarily the entries of $pq^2 A^{-1}$ are all integers, since the linking numbers on $\zz_p \oplus \zz_p$ are all rational with denominator $p$.)  The determinant of this matrix is $\Delta = (pq^2)^2 / \det(A) = -(p^2 q^4)/ p^2 q$, with the minus sign appearing because $A$ is indefinite.  Removing squares,  we have $\Delta = -q \in \ff_p^* / (\ff_p^*)^2$.  

The discriminant of a rank 2 form is the negative of its determinant, so  the discriminant is $\delta = q \in \ff_p^* / (\ff_p^*)^2$.
To conclude the proof, we use  Theorem \ref{witt} in the appendix, stating that a form $\beta$ of rank $n$ is metabolic if and only if $n$ is even and the discriminant satisfies  disc($\beta) = 1 \in  \ff_p^* / (\ff_p^*)^2$.

\end{proof}

\section{Heegaard-Floer obstructions to bounding negative definite forms}\label{sectionnegdef}

\subsection{The Heegaard-Floer correction term}

We begin with a brief review of $Spin ^c$--structures on manifolds.  Recall that there is a canonical 2--fold covering space  $Spin(n) \to  SO (n)$.  The group $Spin^c(n)$ is defined to be the quotient $  Spin(n) \times_{\zz_2} S^1$ where $\zz_2$ acts on $Spin(n)$ by the covering involution and on $S^1$ by multiplication by $-1$.  There is a natural map  $Spin^c(n) \to$ $SO (n)$.  A $Spin^c$--structure on an $n$--manifold $N$ is a lifting of the principal $SO(n)$ bundle associated to the tangent bundle to a principal $Spin^c(n)$ bundle.  The canonical map $Spin^c(n) \to S^1$ associates a complex line bundle $\call(\xi)$ to a $Spin^c$ structure $\xi$ on $M$.  The first Chern class is defined by $c_1(\xi) = c_1(\call(\xi))$.  Here are a few basic facts.

\begin{itemize}

\item The tangent bundle lifts to a $Spin^c$ bundle if and only if there is an integer lifting of $w_2(N)\in H^2(N, \zz_2)$ to $H^2(N)$.

\item If $M$ has a $Spin^c$--structure, then there is a free transitive action of $H^2(N)$ on the set of $Spin^c$--structures.  Furthermore,  $c_1(x \xi) = c_1(\xi) + 2x$ for $x \in H^2(N)$.

\item Every smooth 4--manifold has a $Spin^c$--structure (see, for instance,~\cite{Morgan}).

\end{itemize}

In~\cite{os2}  an invariant $d(M,\spin) \in \qq$ is associated to each 3--manifold $M$ with $Spin^c$--structure $\spin$.  It satisfies   additivity:   
  $d(M_1 \cs M_2, \spin_1 \cs \spin_2) = d(M_1 , \spin_1  )+ d(  M_2,   \spin_2)$.  The main result we use is the following,  Theorem 9.6 in~\cite{os2}.

\begin{theorem}\label{osthm} Let $M$ be a rational homology 3--sphere and fix a $Spin^c$--structure $\spint$ over $M$. Then  for each smooth, negative-definite 4--manifold $W$ with boundary $M$, and for each $Spin^c$--structure $\spin \in Spin^c(W)$ with $\spin|_M = \spint$, we have that $$c_1(\spin)^2 + \beta_2(W) \le 4d(M, \spint).$$

\end{theorem}

We recall the definition of $c_1(\spin)^2$.  Let $z \in H^2(W)$.  There is the exact sequence $H^2(W, M) \to H^2(W) \to H^2(M)$.  Since $H^2(M)$ is torsion, there is a class $z' \in H^2(W,M)$ which maps to $nz \in H^2(W)$ for some integer $n >0$.  We define $z^2 = (z' \cup z) [W,M]/n \in \qq$.  If the intersection form on $H^2(W,M)$ is given by a matrix $A$, then $z^2$ is given by $z^t A^{-1}z$ when  $z$ is written as a vector with respect to the appropriate basis.  

\subsection{Uniform bounds on $c_1^2$}

We will use the following, a proof of which can be found in~\cite{cassels}: 

\begin{theorem}\label{ff} For any given determinant $D$  and integer $n>0$, the set of isometry classes of     symmetric bilinear forms on $\zz^n$ with determinant $D$ is finite.

\end{theorem}

 \begin{corollary}\label{corbound} There is a number $N$ depending only on a pair of positive integers $n$  and $D$,  with the following property:  If $W$ is a 4--manifold  with a positive definite intersection form of rank $n$,  $ |H_1(\partial W)|=D$ and  $\calk' \subset H^2(W)$ is a coset of
$\ker (H^2(W) \to H^2(W, \zz_2))$, then there is an element
$z \in \calk'$ with $z^2 \le N$.

\end{corollary}

 \begin{proof} Let
$\zz^n = \Lambda$,
  and $\Lambda' = \ker (\Lambda \to \Lambda \otimes \zz_2)$.
  By Theorem \ref{ff},  there are finitely many
   isometry classes of positive definite bilinear forms on $\Lambda$ whose determinant divides $D$.
   Each is given by a matrix, say, $A.$
    For each coset of $C$ of $\Lambda'$, we choose a  vector
  $x \in C$
   which minimizes $q_A(x) = x^t A^{-1} x$.
Since $A^{-1}$ has as eigenvalues the reciprocals of the eigenvalues of $A$, it follows that $A^{-1}$ is also positive definite, and so there is  a minimum value.
The maximum of the values of such $q_A(x)$ taken over all finitely many $A$ chosen above and all  finitely many cosets, provides the desired bound $N$.

If $W$ is as hypothesized,  then the intersection form on $W$ is isometric
to one given by an $A$ considered above.
\end{proof}

One can give a simpler proof of this result. The proof we give provides a smaller $N$, which will useful in future applications.

 \subsection{Knots $K$ for which $M(K)$ does not bound negative definite}\label{dnbnd}
 
 \begin{theorem}\label{ndefthm} Let $G$ be a  finite abelian group of  odd order $D$.  Let $\beta$ be  a linking form on $G$.  There exists a knot $K$ such that $(M(K), \lk)  \cong (G, \beta)$  and $M(K)$ does not bound a negative definite 4--manifold $W$ with $\beta_2(W) = 2$.
 \end{theorem}

 \begin{proof}
 Any linking form $(G, \beta)$ on an odd order abelian group $G$ can be diagonalized.  Any one dimensional linking form on an odd order cyclic group is realized as the linking form of the 2--fold branched cover of a 2--bridge knot, a lens space.  Thus, every linking form on an odd order group is realized as the linking form on $M(K)$ for some knot $K$. For more details, see Appendix~\ref{GS}.
 
  For any knot $K$ with $H_1(K) \cong G$ and linking form $\beta$, suppose that $M(K)$ does bound such a manifold $W$ with intersection form $A$.  By Lemma~\ref{githm}, $|\det(A)|$ divides $D$ (since a summand of $H_1(M(K))$ is presented by $A$).  
 
 Choose a $Spin^c$--structure $\spin$ on $W$ and let $\spint$ denote the restriction of $\spin$ to $M(K)$.  For any $x \in H^2(W)$, the $Spin^c$--structure $x\spin$ has $c_1(x\spin) = c_1(\spin) + 2x$.  Thus, the set of all $Spin^c$ structures on $W$ maps via $c_1$ onto a coset of $\calk \subset H^2(W)$.  By Corollary~\ref{corbound} there is a $N>0 $ (depending only on $D$, not on $W$) such that  for some $Spin^c$--structure $\spin'$
on $W
 $, $-N < c_1(\spin')^2  <0$.  (Here we have switched to negative definite from positive definite.)  
Let $\spint'$ denote the restriction of $\spin'$ to $M(K)$.  
 
  By Theorem~\ref{osthm}  we have  
  $4d(M(K), \spint') \ge -N +2$.
   To complete the proof, we show there is a knot $K$ such that this bound fails for all $Spin^c$--structures on $M(K)$.

For any knot $J$ with Alexander polynomial $\Delta_J(t) = 1$,  $H_1(M(J)) = H^2(M(J)) = 0$.  According to~\cite{mo}, there exists such a knot $J$ for which $d(M(J),\spint_0) = \alpha < 0,$ where  $\spint_0$ is the unique $Spin^c$--structure on $M(J)$.  
 
Let $K$ be some knot with   $(M(K), \lk)  \cong (G, \beta)$.  Let $d$ be the maximum value of $d(M(K),\spin)$ taken over all $Spin^c$--structures on $M(K)$.   Now, consider
 $K_n=
 K \cs nJ$.  Each $Spin^c$--structure on $M(K_n) = M(K) \cs_n M(J)$ is the connected sum of a $Spin^c$--structure on $ M(K) $ and a $Spin^c$--structure on each $M(J)$ summand.  Since $M(J)$ has a unique $Spin^c$--structure,  the maximum value of $d(M(K_n), \spin)$ taken over all $Spin^c$--structures is $d + n\alpha$.  We have arranged that  $\alpha$ is negative, so by choosing $n$ large enough, we can ensure that $ 4d(M(K_n), \spin) < -N+2 $.  Thus, $M(K_n)$ cannot bound a negative definite $W$ with $\beta_2(W) = 2$.
 
 \end{proof}
 
 \section{Knots that do not bound Klein bottles}\label{sectionexplicitkknot}
       
       This section presents examples of  knots for which the obstructions developed in Theorem~\ref{pdefthm} and Theorem~\ref{indefthm}  apply to show that the branched cover $M(K)$ does not bound either a positive definite or indefinite 4--manifold of rank 2.  Theorem~\ref{ndefthm} can be applied to build from these examples knots $K$ for which $M(K)$ does not bound a negative definite 4--manifold of rank 2.  Since the construction in Theorem~\ref{ndefthm} does not change the homology, linking form of the 2--fold branched cover, signature, or Arf invariant, the obstructions of Theorems~\ref{pdefthm} and~\ref{indefthm} continue to apply.  Thus, these provide the examples necessary to complete the proof of Theorem~B.

 If $A$ is the symmetrized Seifert matrix for a knot $K$ (that is, $A = V + V^t$, where $V$ is a Seifert matrix for $K$), then the signature of $K$ is $\sigma(K)= \Sign(A),$
 and the determinant is $D(K)=|\det(A)|$. 
 Moreover,  minus the linking form on  the 2--fold branched cover of $S^3$ along $K$ is presented by $A$. Using Theorem \ref{Mil} and Lemma \ref{La}, one sees that the linking form on the double branched cover determines $\sigma(K)$ modulo eight. This relies of work of Taylor \cite{taylor} and Milgram~\cite[Appendix 4]{mh} which is described in the appendix.  (An alternative proof can be based on Lemma~4.1 of~\cite{owst}, which is also developed from the work of Milgram and Taylor.)  One obtains a nice formula for this in the case that $D(K)$ is a prime. 

 \begin{theorem}\label{taylorthm}
  Let $p$  be an odd prime.  
 If $K$ has 2--fold branched cover $M(K)$ with $H_1(M(K)) = \zz_p$ and with 
  minus the linking form taking value $a$ on a generator, then $$\sigma(K) \equiv 2\binom{a}{p} - p -1 \mod 8.$$
\end{theorem}

 \subsection{Levine's theorem on the Arf invariant}

 Levine proved \cite{le2} that the order of $H_1(M(K))$ determines the Arf invariant of $K$:  $\Arf(K) = 0 $ if $D(K)\equiv \pm 1 \mod 8$ and $\Arf(K) = 1$ if  $D(K)\equiv \pm 3 \mod 8$. Murasugi \cite{murasugi2} independently obtained this result shortly thereafter. Kauffman has a nice  proof \cite[Theorem 10.9]{Ka}.  
We perform some arithmetic  to restate this using the greatest integer function, $\lfloor \cdot\rfloor$, as:

\begin{theorem}\label{levinethm}  $\Arf(K) \equiv \lfloor \frac{D(K)+1}{4}\rfloor \mod 2$.
\end{theorem}

\subsection{Two bridge knots and connected sums of lens spaces}

The previous results can   be applied to the connected sum of 2--bridge knots, $  K_{p/a} \cs K_{p/a} \cs K_{q/b}$, with 2--fold branched  cover   the connect sum of lens spaces, $L(p,a) \cs L(p,a) \cs L(q,b)$.   
Minus  
the linking form on $\zz_p \oplus \zz_p \oplus \zz_q$ is given by the diagonal matrix with diagonal $[ \frac{a}{p}, \frac{a}{p},
\frac{b}
{p}]$.
Here we follow the convention 
that $L(p,q)$ is $p/q$--surgery along the unknot in $S^3$.

\begin{theorem} 
Let $p$ and $q$ be odd primes, and let $J$ be a knot whose double branched cover has the same linking form as 
$K_{p/a} \cs K_{p/a} \cs K_{q/b}$.
If $J$ 
bounds a smoothly embedded punctured Klein bottle   $F \subset B^4$  and $W(F)$ is the 2--fold branched cover of $B^4$ branched over $F$, then

\begin{itemize}
\item If $p \equiv 3 \mod 4$ and $q \equiv 1 \in \ff_p^* / (\ff_p^*)^2$, then $W(F)$ is not indefinite.\vskip.05in

\item If 
$4 \lfloor \frac{q +1}{4}\rfloor   + 2\binom{b}{q}   - 2p - q   \equiv 5 \mod 8$, then $W(F)$
 is not positive definite. 

\end{itemize}

\end{theorem}

\begin{proof} The first statement follows from Theorem~\ref{indefthm}.  Notice that the diagonal form $[\frac{a}{p},\frac {a}{p}]$ is not metabolic if $p \equiv 3 \mod 4$.  The second statement follows from Theorem~\ref{pdefthm}, using Theorems~\ref{taylorthm} and~\ref{levinethm}, along with algebraic simplification.
\end{proof}

\begin{proof}[Proof of Theorem B] Consider the knot $3_1 \cs 3_1 \cs -5_2$, which in 2--bridge notation is $K_{3/1} \cs K_{3/1} \cs K_{7/4}$, or any knot with the same linking form. (The naming conventions used here for $3_1$ and $5_2$ are those of~\cite{chaliv}.) It is immediate from the theorem that $K$ cannot bound a Klein bottle $F$ with $W(F)$ indefinite.   
Nor can $K$  bound a Klein bottle $F$ with $W(F)$   positive definite, as we compute that $ \lfloor \frac{7 +1}{4}\rfloor = 2$ and $\binom{4}{7} = 1$. 
Thus such a knot may be modified as in subsection \ref{dnbnd}, so that the new $K$ cannot bound a Klein bottle at all.
Note that the diagonal matrix with diagonal entries $[-3,-21]$ presents  the linking form $\ell(3,1,-1) \oplus \ell(3,1,-1) \oplus \ell(7,1,1)$.
Here we use the notation of Appendix~\ref{GS}: $\ell(p,n,a) $ is the linking form on $\zz_{p^n}$ taking value $a/p^n$ on a generator. Also, according to Wall \cite{wall}, $\ell(3,1,-1) \oplus \ell(3,1,-1) \approx \ell(3,1,1) \oplus \ell(3,1,1)$. Thus the linking form for $K_{3/1} \cs K_{3/1} \cs K_{7/4}$ (and minus this linking form)  has a rank two presentation.
\end{proof}
   
 \section{Casson-Gordon invariants}\label{cgsection}
 The bounds we develop in proving Theorem C depend on Casson-Gordon invariants, as defined in~\cite{cg}.  Here we give a brief review of the definitions, limiting ourselves to the simplest version that applies to the problem at hand.
 
Suppose that $M$ is a closed, oriented, connected 3--manifold and $\rho\co H_1(M) \to \zz_p$ a homomorphism.  Then there exists a 4--manifold $W$ with $\partial W  = kM$ for some $k>0$ and a homomorphism $\tilde{\rho}\co H_1(W) \to \zz_p$ such that $\tilde{\rho}$ restricts to give $\rho$ on each boundary component.   Let $\Sign(W, \rho) $ denote the signature of the intersection form of $H_2(\widetilde{W},\cc)$ restricted to  the $e^{2\pi i/p}$--eigenspace for a generator,  
say $T$, for the $\zz_p$--action on $H_2(\widetilde{W}, \cc)$. Here $\widetilde{W}$ is the $p$--fold cover of $W$ associated to $\tilde{\rho}$, and $T$ acts on $\widetilde{W}$ as follows: If $\tilde \gamma: I \rightarrow \widetilde{W}$ is a lift of a loop $\gamma: I \rightarrow {W}$,  one has that $T^{\rho([\gamma])}\tilde \gamma(0)=  \tilde \gamma(1)$.

The following Casson-Gordon invariant is shown to be well defined in~\cite{cg}, where it is also shown to be a disguised form of the Atiyah-Singer $\alpha$ invariant of the pair $(M^3, \rho)$.

\begin{definition}   $\sigma(M, \rho)  = \frac{1}{k}(\Sign(W,\rho) - \Sign(W))$.
\end{definition}

\begin{definition}   For a knot $K$ with 2--fold branched cover $M$ and $\rho\co H_1(M)\to \zz_p$, define  $\sigma(K, \rho) = \sigma(M, \rho)$.
\end{definition}

\begin{theorem}\label{cgtheorem}  If $M = \partial W$, $H_1(M, \qq) = 0$, and the inclusion $j\co \pi_1(M) \to \pi_1(W)$ is surjective, then for each $\rho$ defined on $M$ that extends to $W$, $$|\sigma(M,\rho) | \le 2\beta_2(W) + 1 + \frac{1}{p-1}\beta_1(\widetilde{M})\in \zz.$$ In the special case that $\rho$ is trivial, $\sigma(M,\rho) = 0$.
\end{theorem}

\begin{proof} Observe first that the surjectivity of $j$ implies that $\beta_1(W) = \beta_1(W,\partial W) =0$, and by duality,  $\beta_3(W) = \beta_3(W,\partial W) =0$.

 Since $\chi(W) = \beta_2(W)  +1$, we have $\chi(\widetilde{W}) = p\beta_2(W) + p$.  The surjectivity of   $\pi_1(M) \to \pi_1(W)$ implies the surjectivity of  $\pi_1(\widetilde{M}) \to \pi_1(\widetilde{W})$, so $\beta_1(\widetilde{W}) \le \beta_1(\widetilde{M})$  and from the long exact sequence, $\beta_1(\widetilde{W},\widetilde{M})=0$. It follows from duality that $\beta_3(\widetilde{W}) = 0.$  Thus, $\chi(\widetilde{W}) = 1 -\beta_1(\widetilde{W}) + \beta_2(\widetilde{W}) = p\beta_2(W) + p$, so $$\beta_2(\widetilde{W}) =   p\beta_2(W) + p +   \beta_1( \widetilde{W})   -1\le   p\beta_2(W) + p +   \beta_1( \widetilde{M})   -1.$$    

Assume now that $\rho$ is nontrivial.  The $1$--eigenspace of the $\zz_p$--action on $H_2(\widetilde{W}, \cc)$ is isomorphic to $H_2(W,\cc)$ and the remaining eigenspaces (of which there are  $(p-1)$) are mutually isomorphic, so each has the same dimension, denoted $\beta_2(W,\rho)$.  We thus have $\beta_2(W,\rho) = (\beta_2(\widetilde{W}) - \beta_2(W))/(p-1)$.  From the previous inequality, it follows that $$\beta_2(W,\rho) \le \beta_2(W) +1 +\frac{1}{p-1} \beta_1(\widetilde{M}).$$ It now follows
 that $$|\sigma(M,\rho) |\le 2\beta_2(W) + 1 + \frac{1}{p-1}\beta_1(\widetilde{M},\rho)$$ 
 as desired. Note that there is no $1$--eigenspace for the $\zz_p$--action on $H_1(\widetilde{M},\cc)$, so $\beta_1(\widetilde{M})$ is divisible by $p-1$.

In the case that $\rho=0$, the $1$--eigenspace of $H_2(\widetilde{W})$ is isometric to $H_2(W)$ with respect to the intersection forms, and thus, the difference of the signatures is 0.
\end{proof}

 \subsection{Extending homomorphisms}   To apply Theorem~\ref{cgtheorem} we need a result concerning the existence of extensions of homomorphisms.  Here is the key lemma, 
  which is related to Lemma 1 of~\cite{gi3}.  Let $\beta_1(M,\zz_p) = \dim_{\zz_p} H_1(M, \zz_p)$.
  
\begin{lemma}\label{metabolizerlemma} Suppose $M = \partial W$, $H_1(M,\qq)=0$ and  $H_1(M)\to H_1(W)$ is surjective. Then the image of the restriction map $\text{hom}(H_1(W),\zz_p) \to \text{hom}(H_1(M), \zz_p)$ is  a subspace of dimension at least $\frac{1}{2}(\beta_1(M,\zz_p) - \beta_2(W))$.    
  \end{lemma}
  
\begin{proof}  It follows immediately from the long exact sequence that $H_1(W,M) =0$.
  
  From Poincar\'e duality and the universal coefficient theorem, $$H_2(W) = H^2(W,M)= \text{hom}(H_2(W,M) , \zz)\oplus \text{Ext}(H_1(W,M),\zz).$$  Since $H_1(W,M) = 0$, the Ext term is trivial, so  we see that $H_2(W)$ is torsion-free and write $H_2(W)= \zz^h$. Switching to $\zz_p$ coefficients, $H_2(W,\zz_p) = (\zz_p)^h \oplus (H_1(W) \otimes \zz_p)$.
  
 Now consider the long exact sequence on cohomology:
 $$ H^1(W,M,\zz_p) \to  H^1(W, \zz_p) \to   H^1(M, \zz_p) \to   H^2(W,M,\zz_p). $$
 The universal coefficient theorem implies the first term  is 0. By Poincar\'e duality with $\zz_p$ coefficients, the last term is isomorphic to $H_2(W, \zz_p)$.  Thus we have the exact sequence
$$0 \to \text{hom}(H_1(W), \zz_p) \to \text{hom}(H_1(M), \zz_p) \to (\zz_p)^h \oplus (H_1(W) \otimes \zz_p ).$$
   If we write $(\zz_p)^m = \text{hom}(H_1(W), \zz_p) \cong  H_1(W) \otimes \zz_p $ and  $(\zz_p)^n  = \text{hom}(H_1(M), \zz_p)$, then this sequence becomes 
 $$0 \to (\zz_p)^m \to(\zz_p)^n   \to (\zz_p)^h \oplus(\zz_p)^m .$$
 It follows that $n - m \le h + m$, so  $m\ge (n-h)/2$.\end{proof}
   
   \section{Ribbon surfaces and obstructions; two-bridge knots}\label{thmbsection}
   
Recall that a smooth surface $F \subset B^4$ is called {\it ribbon} if the restriction to $F$ of the  radial function on $B^4$ has no critical points of index 2.  In particular, a knot $K\subset S^3$ bounds a ribbon surface of Euler characteristic $e$ if and only if some collection of $b$ band moves results in an unlink with $n$ components, where $n-b = e$.  A key property of ribbon surfaces is that the homomorphism $\pi_1(S^3-K) \to \pi_1(B^4 - F)$ is surjective if $F$ is ribbon.  It follows that for the $2$--fold branched covers, the homomorphism $\pi_1(M(K)) \to \pi_1(W(F))$ is also surjective.
 
We now consider a 2--bridge knot  $K = K_{\alpha/\beta}$, and recall that its  2--fold branched cover  is  the lens space $L(\alpha,\beta)$.   The property of lens spaces that we will be using is that for every prime $p$ dividing $\alpha$, the $p$--fold cover $\widetilde{L}$ of $L(\alpha,\beta)$ satisfies $\beta_1(\widetilde{L}) = 0$.

 Fix a prime divisor $p$ of $\alpha$; we consider the set of  $\rho\co \pi_1(L(\alpha,\beta))\to \zz_p$.  To abbreviate, we write $\sigma_{\rho}= \sigma(L,\rho)$.  Also, let $\sigma_{\max} = \max_{\rho \ne 0} \sigma_\rho    $ and $\sigma_{\min} = \min_{\rho \ne 0} \sigma_\rho    $.  If $\sigma_{\max} >0$ and $\sigma_{\min} <0$, the arithmetic becomes somewhat more complicated, with little gain in terms of examples,  so in the following theorem we restrict to the case of $    \sigma_{\min} \ge 0$.

  \begin{theorem} Let $K$ be a 2--bridge knot $K_{\alpha/\beta}$, and  let $p$ be a prime dividing $\alpha$.  Let $\sigma_{\max}$ and $\sigma_{\min}$ be as defined above and assume $   \sigma_{\min}\ge 0$.  Suppose that  $n K= \partial  F $ where $F$ is ribbon and  $h(F) = h$.  Then for some $x\ge \frac{n-h}{2}$ and some $y \le\frac{n+h}{2}$, we have $$  x   \sigma_{\max}   +y  \sigma_{\min}  \le x +y +2h.$$ 
 \end{theorem}

  \begin{proof}

  Let $M = nL$ be the $2$--fold branched cover of $S^3$ branched over $nK$.  Then $H_1(M) =(\zz_\alpha)^n$.  By Lemma~\ref{metabolizerlemma} there is a
   subspace $H$ of $\hom( H_1(M), \zz_p) \cong (\zz_p)^n$ of dimension $\frac{n-h}{2}$ with the property that for all $\rho \in  H$, there is an extension of $\rho$ to
    the $2$--fold branched cover of $B^4$ over $F$.   
  
It follows from Theorem~\ref{cgtheorem} that for nontrivial $\rho \in H$, $$ |\sigma(nK, \rho) | \le 2\beta_2(W) + 1 + \frac{1}{p-1}\beta_1(\widetilde{M}).$$

Form a   $\zz_p$ matrix of size $\frac{n-h}{2} \times n$ using as rows  a set of basis vectors for $H$.  After swapping columns and performing row operations, the first $\frac{n-h}{2} \times \frac{n-h}{2}$ block in that matrix can be made the identity.  Adding the rows yields a vector in $H$ with at least $\frac{n-h}{2}$ entries of 1.   Suppose that exactly $x$ of the entries are 1 and that $y$ other entries are nonzero.  

The value $\sigma_{\max}$ is realized by a character $\rho \in $hom$(H_1(L), \zz_p) \cong \zz_p$  that corresponds to an element $a \in \zz_p$.
Multiply the vector in $H$ constructed above (having $x$ entries 1 and $y$ other entries nonzero) by $a$.  Call the corresponding character $\rho'$.   Then for the associated cover $\widetilde{M}$ we have $\beta_1(\widetilde{M}) = (p-1)(x+y-1)$.  We have already seen that $\beta_2(W) = h$.  Thus, the previous inequality becomes for this particular character
$$|\sigma(nK, \rho') | \le 2h +    x+y .$$ (The additivity of $\sigma(M,\rho)$ under connected sums follows immediately from its definition.) However,  $x \sigma_{\max} + y\sigma_{\min} \le \sigma(nK, \rho')$, completing the proof of the theorem. 
\end{proof}
  
  \begin{corollary}\label{boundcor} If $K$ is a 2--bridge knot and $  \sigma_{\min} \ge 1$, then $$h^r(K) \ge \left(\frac{\sigma_{\max} -1}{\sigma_{\max} +3} \right)n.$$
   \end{corollary}
  
  \begin{proof} From the theorem we have
  $$2h \ge x (\sigma_{\max} -1) + y(\sigma_{\min} -1)\ge \left( \frac{n-h}{2}\right)(\sigma_{\max} -1).$$  The corollary follows immediately.
  
  \end{proof}
             \noindent{\bf Example}  Consider the case of the 2--bridge knot $K = K_{25/2}$.  This is the first example of an algebraically slice knot that is not slice, the 6--twisted double of the unknot. The 2--fold branched cover is $L(25,2)$, and according to~\cite{cg}, the two nontrivial representations to $\zz_5$ yield values $\sigma_{\max} = 5$ and $\sigma_{\min} = 3$. Here (and below) we have  changed the sign as our convention for the orientation of lens spaces is opposite to that used in ~\cite{cg}. Thus, according to the corollary,  $ h \ge  \frac{1}{2}  n.$ \vskip.1in
             
             In general, the bound on $h^r(nK)$ given by Corollary~\ref{boundcor} is significantly less than $n$.  However, the fact that $h^r $ is an integer leads to stronger bounds in some cases.  The following corollary provides examples.

       \begin{corollary}\label{boundcor2}Let $K$ be a 2--bridge knot satisfying $  \sigma_{\min} \ge 1$.  If  $n < \frac{\sigma_{\max}  +3}{4}$ then $h^r(nK) \ge n$.
   \end{corollary} 
    \begin{proof} Using the integrality of $h^r(nK)$, we have   $h^r(nK) \ge n$ if $h^r(nK) > n-1$.  This will be the case if $$\left(\frac{\sigma_{\max} -1}{\sigma_{\max} +3} \right)n > n-1.$$ The corollary follows from simple algebra.
    \end{proof}
    
         \noindent{\bf Example}  Consider the case of the 2--bridge knot $K = K_{13^2/2}$.  For this knot the formula of~\cite{cg} gives    $\sigma_{\max} = 83$ and $\sigma_{\min} = 23$. Thus, according to the corollary,  if $n <  21$, then $h^r(nK) \ge n$.\vskip.1in
             
\section{The topological category}\label{topsection}

A topological locally flat surface $F \subset B^4$ has  a   disk bundle  neighborhood~\cite{freedman-quinn}.  It follows that the 2--fold branched cover $W(F)$ exists and is a manifold.  Our argument that $\chi(W(F)) =   2 - \chi(F)$ was based on $F$ being a subcomplex of a triangulation of $B^4$, which might not be the case.  Instead one can decompose $B^4$ as $N(F) \cup (B^4 - F)$, where $N$ is a disk bundle neighborhood of $F$.  A Mayer-Vietoris argument computes the Euler characteristic of the cover.  The Euler characteristic of the branched cover of $N(F)$ equals the Euler characteristic of $F$.  On $B^4 - F$ we are considering a regular cover, so the Euler characteristic multiplies.  The intersection deformation retracts to a circle bundle, and thus has trivial Euler characteristic. 
It  follows that Theorem \ref{mythm},
 Corollary~\ref{theoremmobius} and Theorem~\ref{theoremklein}  apply in this category.  Furthermore, the Casson-Gordon theory extends to this setting, so Theorem B  also holds in the topological category.  (In the topological locally flat setting, the condition that the surface be ribbon is replaced with the condition that it be {\it homotopy ribbon}; that is, $\pi_1(S^3 - K) \to \pi_1(S^4 - F)$ is surjective.)

The distinction between the two categories is revealed by Theorem D:\vskip.1in

\noindent{\bf Theorem D} {\it There exists a knot $K$ that bounds a topological Mobius band in $B^4$, but does not bound a smooth Mobius band.}\vskip.1in

\begin{proof} Consider a knot $K $ for which $H_1(M(K)) = \zz_3$ and for which the linking form on $H_1(M)$ is 
represented 
 by the $1 \times 1$ matrix 
$(\frac{1}{3})$.
It follows that if $K$ bounds a Mobius band $F$ (in either category) 
then  $\beta_2(W(F)) = 1$ and $W(F)$ is negative definite.  By the  arguments of subsection \ref{dnbnd}, forming the connected sum of $K$ with some Alexander polynomial one knot $J$ yields a 
knot that does not bound a {\em smooth}
Mobius band.  On the other hand, since $J$ is topologically slice, if $K$ bounded a Mobius band,
 then $K \cs J$ would also bound a Mobius band topologically.  The simplest example can be built from the left-handed trefoil, $-3_1$.

\end{proof}

\appendix

\section{Algebraic theory of linking forms} \label{alglK} 
A {\it  linking form} $(G, \beta)$ consists of a  finite abelian group $G$ and a nonsingular symmetric bilinear pairing $\beta\co\! G \times G \to \qq/\zz$.  Nonsingular means that the map $x \to 
\beta
(x, \cdot)$ defines an isomorphism $G \to \text{hom}(G, \qq/\zz)$. Any linking form splits as a direct sum over the $p$--primary summands of $G$.

The applications in this paper  do not require us to consider linking forms on groups of even order. We make the simplifying assumption that $G$ has odd order, and thus we may assume that $p$ is odd.

In the case that $G = \ff_p^n$, where $\ff_p$ is the field with $p$--elements, $p$ prime, we can view the linking form as taking values in $\ff_p$ by mapping the rational number $\frac{a}{p}$ to $a \in \ff_p$.  A choice of basis yields a  matrix representation of the linking form; this is a nonsingular symmetric square matrix  $Q$ with entries in $\ff_p$.  The discriminant of $\beta$, disc($\beta$), is defined in terms of the determinant of $Q$: disc$(\beta)= (-1)^{n(n-1)/2}\det(Q) \in \ff_p^* / (\ff_p^*)^2$.  Quotienting by squares  ensures that  the discriminant depends only on the isometry class of $\beta$ and not the choice of basis.    If $\beta$ is {\it metabolic} (that is, if there is a half-dimensional summand on which $\beta$ vanishes), then $\text{disc}(\beta) = 1$.  More generally, using \cite[Chapter IV Lemma 1.5]{mh} (or see~\cite[Theorem~B.5]{liv2}) it follows  that:

\begin{theorem}\label{witt} A nonsingular bilinear form $\beta$ on $\ff_p^n$ with $p$ odd is metabolic if and only if $n$ is even and disc$(\beta)  = 1 \in  \ff_p^* / (\ff_p^*)^2$.
\end{theorem}

Suppose we have a nonsingular bilinear  form $b$ on $L$, a free $\zz$--module of rank  $n$. Then there is a dual lattice $L^{\#} 
= \{ v \in L \otimes \qq\  | \ b(v,w) \in \zz, \ \forall \ w \in L\}$. The form $b$ extends to a rational valued form on  $L \otimes \qq$ which then restricts to a rational valued form on $L^{\#}.$  The quotient  $L^{\#} /L$ is a finite abelian group. If one picks a basis for $L$,   $b$ is given by  a symmetric nonsingular integer matrix  $A$  over $\zz$. One has that $A$ is a presentation matrix for $L^{\#} /L$. There is a linking form defined on $L^{\#} /L$ by $\beta (\bar x,\bar y)= b(x,y)  \pmod{\zz} $. In this situation, $A$ defines a  1--1  map   sending $L^{\#}$ to $L$ and $L$ to $A(L)$.  Thus we have an isomorphism to $L^{\#} /L \approx L/A(L) $.  When $\beta$ is transfered by this isomorphism  to  $L/A(L)$, we have $\beta (\bar w,\bar z)= w^t A^{-1} z \pmod{\zz} $. We say $b$ or $A$ {\it presents} $\beta,$ and write $\beta_A$ for the linking form presented by $A.$ The linking form $\beta_A$, having values in $\qq/\zz$, has matrix representation given by    $A^{-1}$.

\section{Quadratic Linking forms} \label{sectionquadratic}
We continue to let $G$ be a finite abelian group of odd order. 
A quadratic linking form on 
$G$ is a function $q:G \to \qq/\zz$ such that  $q(-g)=q(g)$ holds for all $g \in G$ and $q(g+h)-q(g)-q(h)$ defines a linking form on $G$.
If $q$ is a quadratic form on $G$, let $\beta_q$ denote the linking form defined by $\beta_q(g,h)= q(g+h)-q(g)-q(h)$.
If $q$ and $q'$ are quadratic linking forms, and $\beta_q=\beta_{q'}$, then $q-q'$ is a linear map to $\{0,\frac 12 \}\subset \qq/\zz$.
In this case, $q=q'$ as $G$ has odd order.

Suppose that $\beta$ is a linking form on  $G$.  Let $2^*$ denote a multiplicative inverse to  $2$ in $G$.  Define a quadratic form $q_\beta$ on $G$ by 
$q_\beta(g) = \beta(2^* g,g)$. 
 One has that $\beta(g,h) =  q_\beta(g+h)-q_\beta(g)-q_\beta(h)$. 
Thus $\beta = \beta_{q_\beta}$. 
This defines a bijection between the set of linking forms and the set of quadratic linking forms on a given odd order group $G$.

 \section{Gauss Sums}\label{GS}
Associated to the quadratic form 
$q$, we have the Gauss  sum considered in \cite{taylor}  $$\Gamma(q) = \frac{1}{\sqrt{|G|} }\sum_{g \in G} \exp(2 \pi i q(g)).$$   The sum is necessarily a fourth root of unity.   
We let  $\begin{pmatrix}  \alpha \\ p  \end{pmatrix}$ denote the  Legendre 
symbol which is $\pm 1 $ depending on whether $a$ is or is not a quadratic residue modulo an odd prime $p$.
Taylor proves the following result.

\begin{lemma}\label{ta} Let $q(p,n,a)$ denote the quadratic form on $\zz_{p^n}$ which sends a generator to $a$.
\[
\Gamma(q(p,n,a))
=  \begin{cases}
 1 &\text{if $n$ is even },\\
\begin{pmatrix} a\\ p  \end{pmatrix}& \text{if $p=1 \pmod 4$ and  $n$ is odd}\\
 i \begin{pmatrix}a\\ p  \end{pmatrix} & \text{if $p=3 \pmod 4$ and  $n$ is odd}.
\end{cases}\]
\end{lemma}

It will be convenient to have a formula for this Gauss sum as it depends on the associated linking form, so we define 
 $$\Lambda(\beta)= \Gamma(q_\beta).$$
 This satisfies $\Lambda(\beta_1 \oplus \beta_2) =\Lambda(\beta_1  )\Lambda(  \beta_2)$.
 The next theorem follows from a result of 
Milgram \cite[Appendix 4]{mh}.

\begin{theorem}\label{Mil}  If $A$ is a square symmetric integral  matrix with even numbers on the diagonal and odd determinant, then  $$\Lambda({\beta_A})=  \exp\left(2 \pi i \frac{\Sign(A)}{8} \right) .$$
\end{theorem}

Let $p$ denote an odd prime. Let $\ell(p,n,a)$ denote  the form on ${\zz}_{p^n}$ given by $\ell(p,n,a)(1,1)=a/p^n.$   Notice that $q_{\ell(p,n,a)} = q(p,n,a/2)$.
Every linking form on an odd order group can be written as a direct sum of such forms \cite{wall}. 
For an odd prime $p$, 2 is a quadratic residue modulo $p$ if and only if $ p \equiv \pm 1 \pmod{8}
$.  Using Lemma~\ref{ta} we have:
\begin{lemma}\label{La}
\[
\Lambda(\ell(p,n,a))
=  \begin{cases}
 1 &\text{if $n$ is even },\\
\begin{pmatrix} a\\ p  \end{pmatrix}& \text{if $p=1 \pmod 8$ and  $n$ is odd}\\
 - i \begin{pmatrix}a\\ p  \end{pmatrix} & \text{if $p=3 \pmod 8$ and  $n$ is odd},\\
 -\begin{pmatrix}a\\ p  \end{pmatrix} & \text{if $p=5 \pmod 8$ and  $n$ is odd},\\
i \begin{pmatrix}a\\ p  \end{pmatrix} & \text{if $p=7 \pmod 8$ and  $n$ is odd}.\\
\end{cases}\]
\end{lemma}

\section{Intersection forms of 4--manifolds}

For any finitely generated abelian group $G$, we write $\bar{G} = G/\text{Torsion}(G)$.   For instance, $\bar{H}^2(X)$ will denote the quotient of the second cohomology of a space by its torsion subgroup.

For a compact $4$--manifold $W$,  there is a symmetric pairing induced by the cup product and evaluation on the fundamental class:  $\bar{H}^2(W) \times \bar{H}^2(W,\partial W) \to \zz$.
 Let $\iota: (W,\emptyset) \rightarrow (W, \partial W)$ denote the inclusion. 
 The map 
 $\iota^*\co \bar{H}^2(W, \partial W) \to \bar{H}^2(W)$ 
 induces a pairing $\bar{H}^2(W, \partial W) \times \bar{H}^2(W, \partial W) \to \zz$ which we denote by $\left< \cdot , \cdot \right>$.  Let $\calb$ be a basis for $\bar{H}^2(W,\partial W)$ and let $A$ be a matrix representing the intersection form with respect to $\calb$. By Poincar\'e duality and universal coefficients, there is an isomorphism $\bar{H}^2(W) \cong \text{hom}(\bar{H}^2(W, \partial W), \zz)$. Let $\calb'$ be the dual basis for $\bar{H}^2(W)$.  With respect to these two bases, the map $\iota^*$ is represented by the  matrix $A$.

Suppose that $H^2(\partial W, \qq) \cong  H_1(\partial W, \qq) = 0$ and let $c  \in H^2(W)$.  The square $c ^2$ is defined as follows.  Under restriction, $c $ maps to a torsion element in $H^2(\partial W)$.  Thus, there is an $x \in H^2(W, \partial W)$ which maps to    $n c $ for some integer $n$.  Set $c ^2 = \left< \bar{x} ,\bar{ x}\right> /n^2 \in \qq$. It is straightforward to check that this does not depend on the choices made, and depends only on  $\bar{c}  \in \bar{H}^2(W)$.  If $\bar{c} $ is expressed in terms of the basis $\calb'$, then $\bar{c} ^2$ is given by 
 the rational valued quadratic form determined by the matrix $A^{-1}$.  (Note that since $H^2(\partial W)$ is finite, $A$ represents a surjective map  and thus is invertible.)

Continuing with the assumption that $H_1(\partial W, \qq) = 0$, there is an {\it intersection form}  $\left< \cdot , \cdot \right>\co  \bar{H}_2(W) \times \bar{H_2}(W) \to \zz$ defined using duality to identify $\bar{H}_2(W)$ with $\bar{H}^2(W, \partial W)$.  This form is thus given by the matrix $A$ defined above, with respect to the corresponding choice of basis. Similarly, the map $\iota_* :\bar{H}_2(W) \to \bar{H}_2(W, \partial W)$ is given by $A$.

\section{Linking forms of 3--manifolds}\label{giapp}

  If a rational homology sphere $M$ is the  boundary of a 4--manifold $W$ and $H_1(W)=0$, with intersection form given by a matrix $A$, then minus the linking form on $H_1(M)$ can be identified with $\beta_A$. 
Let $L$ be the lattice $H_2(W,\zz)$  in $H_2(W,\qq)$, and $L^{\#} \subset H_2(W,\qq)$ be the dual lattice with respect to the intersection form on $H_2(W,\zz)$. Then $L^\#$ is carried to $H_2(W,M)\subset H_2(W,M,\qq)$ by $\iota_*$ and $L$ is carried to the image 
$\iota_*: H_2(M,\zz) \rightarrow H_2(W,M,\zz)$.  The linking form on $H_1(M)$ is  presented by the intersection form on 
$H_2(W)$, in the sense of Appendix \ref{alglK}.

More generally, if we no longer assume $H_1(W)=0$,  we have the following result from~\cite{gi3}.

\begin{lemma}\label{githm}
Let $F$ be a  maximal free summand   $H_2(W, \partial W)$.  Then $H_1(\partial W)$ splits as a direct sum $F^* \oplus G$, where $F^*$ is the image of $F$.  The linking form on $G$ is metabolic, and   
the intersection form on $\bar H_2(W)$ 
 presents  
 minus  
the  linking form of $\partial W$ restricted to $F^*$.  
The splitting $F^* \oplus G$  is orthogonal with respect to the linking form.
\end{lemma}
 
\vskip.1in
  
\newcommand{\etalchar}[1]{$^{#1}$}

\end{document}